\documentclass{amsart}

\usepackage{amsmath,amssymb,amsfonts,mathtools}
\usepackage{xcolor}
\usepackage{verbatim}
\usepackage{hyperref}
\usepackage{svg}

\newtheorem{theorem}{Theorem}[section]
\newtheorem{corollary}{Corollary}[section]
\newtheorem{lemma}[theorem]{Lemma}

\theoremstyle{definition}

\theoremstyle{remark}

\numberwithin{equation}{section}

\begin{document}

\title[On Taylor-like Estimates for $L^{2}$ Polynomial Approximations]{On Taylor-like Estimates for Chebyshev Series Approximations of $e^x$}


\author[Ale\v{s} Wodecki, Shenyuan Ma]{Ale\v{s} Wodecki, Shenyuan Ma}
\address{Czech Technical University in Prague}
\curraddr{Technick\'{a} 2, Prague, Czech Republic}
\email{wodecki.ales@fel.cvut.cz, shenyma@fel.cvut.cz}
\thanks{}

\subjclass[2020]{Primary }

\date{January 12, 2024}

\begin{abstract}
Polynomial series approximations are a central theme in approximation theory due to their utility in an abundance of numerical applications. The two types of series, which are featured most prominently, are Taylor series expansions and expansions derived based on families of $L^{2}-$orthogonal polynomials on bounded intervals. Traditionally, however, the properties of these series are studied in an isolated manner and most results on the $L^{2}-$based approximations are restricted to the interval of approximation. In many applications, theoretical guarantees require global estimates. Since these often follow trivially from Taylor's remainder theorem, Taylor's series is often favored in these applications. Its suboptimal numerical performance on compact intervals, however, ultimately results in such numerical algorithms being suboptimal. We show that "Taylor-like" bounds valid outside of the interval of approximation may be derived for the important case of $e^{x}$ and Chebyshev polynomials on $[-1,1]$ and provide links to substantial numerical applications.

\end{abstract}

\maketitle

\section{Introduction}
Traditionally, Taylor expansions have been of central interest in mathematics. Efforts to find generalizations can be traced back to the seminal paper of Widder \cite{widder27}, where a generalization from the monomial basis to a basis of general functions is discussed. Over the years, other types of generalizations, which include links to asymptotic series and generalized functions, have arisen \cite{stan96, Odibat07, masjed18}. On the other hand, series derived based on projections onto an orthogonal basis in a given $L^{2}$ space, due to their success in applications, have been studied heavily by the approximation theory community \cite{Carlson74, zhong00, costin16}. Due to the inherent locality of the results, these series are seldom used in situations, in which an unbounded interval is present, as one rather makes use of a simple form of the remainder term, which is provided by Taylor's expansion. Motivated by the preceding, an estimate commonly available only for Taylor's expansion by virtue of its simple remainder form is derived for the Chebyshev series expansion.

To further motivate the upcoming discussion, let $e^{x}$ be expressed using its Maclaurin series centered at $0$ as
\begin{equation}\label{eq_taylor_of_ex}
e^{x}=\sum_{n=0}^{N}\frac{x^{n}}{n!}+\frac{x^{n+1}e^{\xi}}{\left(n+1\right)!},    
\end{equation}
where $N\in\mathbb{N}$ and $\xi\in\left(0,x\right)$ for $x>0$ and $\xi\in\left(x,0\right)$ for $x<0$. By inspecting the remainder term in \eqref{eq_taylor_of_ex} one can readily derive the following upper and lower bound for the exponential
\begin{equation}\label{eq_bounds_of_exp_Taylor}
\sum_{n=0}^{N}\frac{x^{n}}{n!}\leq e^{x}\leq\sum_{n=0}^{N+1}\frac{x^{n}}{n!}\text{ for any odd }N\in\mathbb{N}\text{ and }x<0.
\end{equation}
Because of the prominent role of the exponential, these bounds find frequent use in literature \cite{bakshi2023learning, Monera12, bader19, Nilsson15}, resulting in guarantees of convergence and accuracy. The quality of these bounds and the numerical performance of the resulting algorithm is also affected by the precision of the approximation on a bounded interval in the uniform norm \cite{bakshi2023learning, wodecki2024learningquantumhamiltonianstemperature}. This gives strong motivation to derive such bounds for non-Taylor $L^2$-based approximations.

However, bounds of the type \eqref{eq_bounds_of_exp_Taylor} for non-Taylor expansions have not been discussed in literature as approximation theory typically limits itself to the interval on which the approximation is constructed \cite{xie23,Madan87, zhang21, xiang10}. A motivating example of such an expansion is the Chebyshev expansion with respect to the space $L^{2}\left(\left[-1,1\right],\frac{1}{\sqrt{1-x^{2}}}dx\right)$. This approximation is known to provide a near min-max uniform norm approximation of a smooth function on $\left[-1,1\right]$, which turns out to be better than the Taylor approximation with respect to the supremum norm \cite{sach14, Yevick05}. Recalling the identity theorem \cite{stein03}, which states that any two analytical functions, which agree on an accumulation point, must agree in all of $\mathbb{C}$, one may notice that any polynomial series approximating $e^{x}$ on an interval of arbitrary size must converge to this function on all of $\mathbb{R}$. Thus, for any polynomial series, it is reasonable to ask the question: what is the behavior of this approximation outside of the interval of approximation and can some claims about the relationship between the approximation and the original function be made?

Due to the nature of non-monomial polynomials bases, it is not reasonable to expect the estimate to follow from a simple reminder term as in \eqref{eq_taylor_of_ex}. Nevertheless, the utilization of the properties of expansion coefficients and Chebyshev polynomials, one is able to arrive at an analogous estimate for the Chebyshev series (Theorem \ref{thm:main-theorem}). A simple graphical illustration of the theorem can be found in Figure \ref{fig:exp-cheb}.

\begin{figure}
    \centering
    \includegraphics[width=0.8\textwidth]{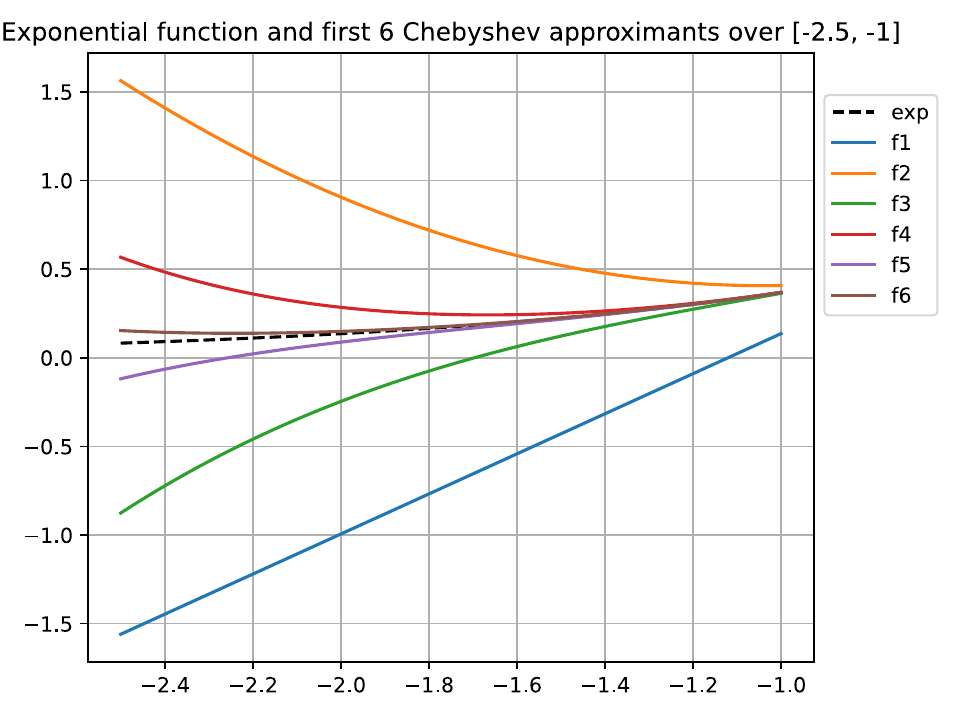}
    \caption{A numerical example, which shows that the odd and even degrees indeed do provide a lower and upper bound outside of the domain $[-1,1]$ of $L^2$ approximation, respectively.}
    \label{fig:exp-cheb}
\end{figure}

\section{Useful Identities}
The following proofs hinge on a couple of well known properties of Chebyshev Polynomials and Bessel functions, the following Lemma summarizes these.
\begin{lemma}
Let 
\begin{equation}
I_{v}\left(z\right)\coloneqq i^{-v}J_{v}\left(ix\right)=\sum_{m=0}^{\infty}\frac{1}{m!\Gamma\left(m+v+1\right)}\left(\frac{x}{2}\right)^{2m+v}    
\end{equation}
denote the modified Bessel function of the first kind and 
\begin{align}
T_{n}\left(z\right)\coloneqq\frac{n}{2}\sum_{k=0}^{\left\lfloor \frac{n}{2}\right\rfloor }\left(-1\right)^{k}\frac{\left(n-k-1\right)!}{k!\left(n-2k\right)!}\left(2z\right)^{n-2k} ,\\
U_{n}\left(z\right)\coloneqq\sum_{k=0}^{\left\lfloor \frac{n}{2}\right\rfloor }\left(-1\right)^{k}\left(\begin{array}{c}
n-k\\
k
\end{array}\right)\left(2z\right)^{n-2k}
\end{align}

denote the Chebyshev polynomials of the first kind and second kind, respectively. Then the following identities and estimates hold
\begin{align}
\label{eq_chebyshev_derivative}
T_{n}^{\prime}\left(x\right)=nU_{n-1}\left(x\right), &&\text{where }x\in\mathbb{R},\\
\label{eq_chebyshev_1st_second}
U_{n}\left(x\right)=\begin{cases}
2\sum_{\text{odd }j}^{n}T_{j}\left(x\right) & \text{for odd }j,\\
2\left(\sum_{\text{even }j}^{n}T_{j}\left(x\right)\right)-1 & \text{for even }j, 
\end{cases} &&\text{where }x\in\mathbb{R},\\
\label{eq_mod_bessel_id}
vI_{v}\left(x\right)=\frac{x}{2}\left(I_{v-1}\left(x\right)+I_{v+1}\left(x\right)\right) &&\text{ for }x>0,v>-\frac{1}{2},  \\
\label{eq_bessel_lower}
I_{v}\left(x\right)>\frac{1}{\Gamma\left(v+1\right)}\left(\frac{x}{2}\right)^{v}\text{ for }x>0,v>-\frac{1}{2}, \\
\label{eq_bessel_upper}
I_{v}\left(x\right)<\frac{\cosh x}{\Gamma\left(v+1\right)}\left(\frac{x}{2}\right)^{v} &&\text{ for }x>0,v>-\frac{1}{2}.
\end{align}
\end{lemma}
\begin{proof}
\eqref{eq_chebyshev_derivative} - \eqref{eq_chebyshev_1st_second} follow by direct computation from the definitions, the identity \eqref{eq_mod_bessel_id} can be derived by applying trigonometric identities \cite{abra74} and \eqref{eq_bessel_upper}, \eqref{eq_bessel_lower} are proven in \cite{LUKE197241}.    
\end{proof}

Before the main theorem is proven, a simple estimate that finds use in the subsequent root isolation process is provided.

\begin{lemma}[Modified Bessel function estimate]\label{lem_mod_bessel_function_est}
    The following esimate holds for modified Bessel functions of the first kind
    \begin{equation}\label{eq_general_estimate_for_ratio_of_bes}
        \frac{I_{n+1}\left(x\right)}{I_{n}\left(x\right)}\leq\frac{\cosh\left(x\right)x}{2\left(n+1\right)},
    \end{equation}
    where $x>0$ and $n\in\mathbb{N}$.
\end{lemma}
\begin{proof}
Applying the estimates \ref{eq_bessel_lower} and \ref{eq_bessel_upper} to overestimate the fraction yields
\begin{equation}
\frac{I_{n+1}\left(x\right)}{I_{n}\left(x\right)}\leq\frac{\frac{\cosh x}{\left(n+1\right)!}\left(\frac{x}{2}\right)^{n+1}}{\frac{1}{n!}\left(\frac{x}{2}\right)^{n}}\leq\frac{\cosh\left(x\right)x}{2\left(n+1\right)}.    
\end{equation}
\end{proof}

\begin{corollary}
The result \eqref{eq_bessel_upper} is commonly simplified as $\cosh\left(x\right)<e^{x}$, resulting in a less sharp estimate for the fraction of subsequent Bessel functions for small $x$
\begin{equation}\label{eq_bessel_not_sharp_enough}
\frac{I_{n+1}\left(x\right)}{I_{n}\left(x\right)}\leq\frac{e^{x}x}{2\left(n+1\right)}.
\end{equation}
It will be shown that the sharper estimate can be leveraged along with Theorem \ref{thm_reduction_theorem} to achieve better results. Specializing the result of Lemma \ref{lem_mod_bessel_function_est} to $x=1$ one obtains
\begin{equation}\label{eq_bessel_sharp_enough}
\frac{I_{n+1}\left(1\right)}{I_{n}\left(1\right)}\leq\frac{4}{5\left(n+1\right)},
\end{equation}
which is notably sharper than the bound $\frac{e}{2\left(n+1\right)}$, which follows from \eqref{eq_bessel_not_sharp_enough}.
\end{corollary}

The equioscillation theorem does not permit a bound that covers $\left(-\infty,0\right)$ as this would introduce a contradiction with the near min-max quality of the Chebyshev approximation, therefore the following theorem is derived for $\left(-\infty,-1\right)$.

\section{Proof of Main Theorem}

We start by proving a lemma, which is the stepping stone to the final proof.

\begin{lemma}[Reduction theorem]\label{thm_reduction_theorem} Let $N\in \mathbb{N}$ and a let 
\begin{equation}\label{eq_cheby_approx_of_exponential}
f_{n}=\sum_{n=0}^{N}a_{n}T_{n}\left(x\right)   
\end{equation}
be the $N$-th Chebyshev approximation on $L^{2}\left(\left[-1,1\right],\frac{1}{\sqrt{1-x^{2}}}dx\right)$. Then 
\begin{equation}\label{eq_upper_bound_cheb}
e^{x}\leq\sum_{n=0}^{N}a_{n}T_{n}\left(x\right)\end{equation}
holds for even $n$ on $\left(-\infty,-1\right)$ if the polynomial
\begin{equation}\label{eq_key_polynomial_for_thm}
G_{N}\left(x\right)=I_{N+1}\left(1\right)U_{N-1}\left(x\right)+I_{N}\left(1\right)U_{N-2}\left(x\right)-I_{N}\left(1\right)+I_{N}\left(1\right)T_{N}\left(x\right)
\end{equation}
is positive on $\left(-\infty,-1\right)$ and 
\begin{equation}
e^{x}\geq\sum_{n=0}^{N}a_{n}T_{n}\left(x\right)\end{equation}
holds for odd $n$ on $\left(-\infty,-1\right)$ if $G_{N}\left(x\right)$ is negative on $\left(-\infty,-1\right)$.
    
\end{lemma}
\begin{proof}
Let $N$ be even and define
\begin{equation}
g_{N}\left(x\right)\coloneqq\sum_{n=0}^{N}a_{n}T_{n}\left(x\right)-e^{x}.    
\end{equation}
Since $N$ is even, $g_{N}\left(x\right)\rightarrow+\infty$ when $x\rightarrow+\infty$ and thus the following conditions ensure that \eqref{eq_upper_bound_cheb} holds:
\begin{itemize}
    \item  $g_{N}\left(-1\right)\geq0$,
    \item $g_{N}\left(x_{0}\right)$ is positive at all extremal points $x_{0}<-1$.
\end{itemize}

Since the chebyshev polynomials at $x=-1$ satisfy $T_{n}\left(-1\right)=\left(-1\right)^{n}$ one may write

\begin{equation}
e^{-1}=a_{0}+\sum_{n=1}^{\infty}a_{n}\left(-1\right)^{n}=a_{0}+\sum_{n=1}^{\infty}\left[-a_{2n-1}+a_{2n}\right],    
\end{equation}
Using the fact that the Chebyshev coefficients in \eqref{eq_cheby_approx_of_exponential} satisfy
\begin{equation}\label{eq_chebyshev_coefficients}
a_{0}=I_{0}\left(1\right),\quad a_{n}=2I_{n}\left(1\right)\text{ for n\ensuremath{\geq1}}
\end{equation}
and $I_{n}\left(x\right)\geq I_{n+1}\left(x\right)$ one arrives at
\begin{equation}
e^{-1}\leq a_{0}+\sum_{n=1}^{\frac{N}{2}}\left[-a_{2n-1}+a_{2n}\right] \quad \text{for every odd $N$}.
\end{equation}
This implies that $g_{N}\left(-1\right)\geq0$ for any even $N$.

To find the condition for the positivity of $g_{N}\left(x_{0}\right)$ at all extremal points $x_{0}<-1$ suppose that $y$ is an extremal point. Then 
\begin{equation}
0=g_{N}^{\prime}\left(y\right)=\sum_{n=0}^{N}a_{n}T_{n}^{\prime}\left(y\right)-e^{y}=\sum_{n=0}^{N-1}\left(n+1\right)a_{n+1}U_{n}\left(y\right)-e^{y},
\end{equation}
where the relationship \eqref{eq_chebyshev_derivative} was used. To make the association with the original series the relationship \eqref{eq_chebyshev_1st_second} is applied, yielding
\begin{align}
f_{n}^{\prime}\left(x\right)=&\sum_{n=0}^{\frac{N-2}{2}}\left(2n+1\right)a_{2n+1}U_{2n}\left(x\right)+\sum_{n=1}^{\frac{N}{2}}2na_{2n}U_{2n-1}\left(x\right) \\
=&\sum_{n=0}^{\frac{N-2}{2}}\left(2n+1\right)a_{2n+1}2\left(\sum_{j=0,j\text{ even}}^{n}T_{j}\left(x\right)-1\right)+\sum_{n=1}^{\frac{N}{2}}2na_{2n}2\sum_{j=0,j\text{ odd}}^{n}T_{j}\left(x\right)\\
=&\sum_{n=0}^{\frac{N-2}{2}}\left[\underbrace{\left(\sum_{j=0,j\text{ even}}^{2n}2\left(2n+1\right)a_{2n+1}T_{j}\left(x\right)\right)}_{I.}-\underbrace{2\left(2n+1\right)a_{2n+1}}_{II.}\right]
\\ & +\underbrace{\sum_{n=1}^{\frac{N}{2}}4na_{2n}\sum_{j=0,j\text{ odd}}^{2n}T_{j}\left(x\right)}_{III.}.
\end{align}
Rewriting each of the terms using \eqref{eq_chebyshev_coefficients} and \eqref{eq_mod_bessel_id} leads to
\begin{align}
I.=&\sum_{j=0,j\text{ even}}^{N-2}T_{j}\left(x\right)\sum_{\frac{j}{2}=n}^{\frac{N-2}{2}}4\left(2n+1\right)I_{2n+1}\left(1\right) \\
=&\sum_{j=0,j\text{ even}}^{N-2}T_{j}\left(x\right)2\sum_{\frac{j}{2}=n}^{\frac{N-2}{2}}\left(I_{2n}\left(1\right)-I_{2n+2}\left(1\right)\right),
\\
=&\sum_{j=0,j\text{ even}}^{N-2}2\left(I_{j}-I_{N}\right)T_{j}\left(x\right)\\
II.=&-\sum_{n=0}^{\frac{N-2}{2}}\left(I_{2n}\left(1\right)-I_{2n+2}\left(1\right)\right)=-I_{0}+I_{N}, \\
III.=&2\sum_{j=1,j\text{ odd}}^{N-1}T_{j}\left(x\right)\sum_{\frac{j+1}{2}=n}^{\frac{N}{2}}2\left(2n\right)I_{2n+1}\left(1\right) \\
=&2\sum_{j=1,j\text{ odd}}^{N-1}T_{j}\left(x\right)\sum_{\frac{j+1}{2}=n}^{\frac{N}{2}}\left(I_{2n-1}\left(1\right)-I_{2n+1}\left(1\right)\right) \\
=& 2\sum_{j=1,j\text{ odd}}^{N-1}\left(I_{j}\left(1\right)-I_{N+1}\left(1\right)\right)T_{j}\left(x\right).
\end{align}
Combining these results yields
\begin{align}
I.+II.+III.=&\sum_{j=0}^{N-1}a_{j}T_{j}\left(x\right)-I_{N+1}\left(1\right)2\sum_{j=1,j\text{ odd}}^{N-1}T_{j}\left(x\right)  \\
&-\sum_{j=0,j\text{ even}}^{N-2}2I_{N}\left(1\right)T_{j}\left(x\right)+I_{N}\left(1\right)T_{0}\left(x\right),
\end{align}
which shows that 
\begin{equation}\label{eq_identity_for_derivative_at_ext}
0=g_{N}^{\prime}\left(y\right)=g_{N}\left(y\right)-G_{N}\left(y\right),,    
\end{equation}
where 
\begin{equation}\label{eq_final_identity}
G_{N}\left(y\right)\coloneqq I_{N+1}\left(1\right)2\sum_{j=1,j\text{ odd}}^{N-1}T_{j}\left(y\right)+I_{N}\left(1\right)2\sum_{j=0,j\text{ even}}^{N-2}T_{j}\left(y\right)-I_{N}\left(1\right)+I_{N}\left(1\right)T_{N}\left(y\right)  
\end{equation}
and $y$ is the extremal point. Applying \eqref{eq_chebyshev_1st_second} and rearranging \eqref{eq_identity_for_derivative_at_ext} yields
\begin{equation}
g_{N}\left(y\right)=G_{N}\left(y\right)=I_{N+1}\left(1\right)U_{N-1}\left(y\right)+I_{N}\left(1\right)U_{N-2}\left(y\right)-I_{N}\left(1\right)+I_{N}\left(1\right)T_{N}\left(y\right).    
\end{equation}
Thus, if $G_{N}$ is positive on $\left(-\infty,-1\right)$ it is guaranteed that $g_{N}$ is positive in all extremal points of the interval and therefore $e^{x}\leq\sum_{n=0}^{N}a_{n}T_{n}\left(x\right)$. The proof for odd $N$ is analogous.
\end{proof}

The preceding lemma is utilized to finally prove that we have a bound on $e^{x}$ valid over all of $(-\infty,-1)$.
\begin{theorem}(Chebyshev bound on $e^{x}$)\label{thm:main-theorem} The estimate
\begin{equation}\label{eq:main-bound}
\sum_{n=0}^{2N-1}a_{n}T_{n}\left(x\right)\leq e^{x}\leq\sum_{n=0}^{2N}a_{n}T_{n}\left(x\right)    
\end{equation}
holds for any $N \in \mathbb{N}$ and any $x \in (-\infty,-1)$.
\end{theorem}

\begin{proof}
An elementary proof of the positivity of polynomials in \eqref{eq_key_polynomial_for_thm} is given. To deal with the alternating signs, we prove the positivity of $(-1)^NG_N(x)$ on $(-\infty,-1)$ instead. We need to recall two additional facts about Chebyshev polynomials:
\begin{subequations}
    \begin{itemize}
        \item the Chebyshev polynomials relate the trigonometric functions and their value evaluated at positive integer multiples of arguments. Similar relations also hold for hyperbolic functions, more succinctly 
        \begin{align}
        T_n(\cosh(\theta))&=\cosh(n\theta),\\
        U_n(\cosh(\theta))&=\frac{\sinh((n+1)\theta)}{\sinh(\theta)}.
        \end{align}
        \item the Chebyshev polynomials have alternating sign symmetries, namely \begin{align}
            T_n(-x)&=(-1)^nT_n(x),\\
            U_n(-x)&=(-1)^nU_n(x).
        \end{align}
    \end{itemize}
\end{subequations}

Keeping the above in mind, \eqref{eq_key_polynomial_for_thm} is multiplied by $(-1)^N$ and alternating symmetries are applied to obtain \begin{equation}
    (-1)^N G_N(x)= -I_{N+1}U_{N-1}\left(-x\right)+I_{N}U_{N-2}\left(-x\right)-I_{N}(-1)^N+I_{N}T_{N}\left(-x\right).
\end{equation}
For notational simplicity we just write $I_N$ for $I_N(1)$. Since $x\in(-\infty,-1)$, let's introduce the substitution $x=-\cosh(t)$ where $t\in(0,+\infty)$ and this gives : 
\begin{equation}
    \begin{split}
        (-1)^N G_N\left(-\cosh(t)\right) &= -I_{N+1}U_{N-1}\left(\cosh(t)\right)+I_{N}U_{N-2}\left(\cosh(t)\right)\\
            &-I_{N}(-1)^N+I_{N}T_{N}\left(\cosh(t)\right)\\
            &= -I_{N+1}\frac{\sinh(Nt)}{\sinh(t)}+I_N\frac{\sinh\left((N-1)t\right)}{\sinh(t)}\\
            &-I_N(-1)^N+I_N\cosh(Nt)
    \end{split}
\end{equation}
We multiply both sides by $\sinh(t)$ to clear out the fractions and note that $\sinh(t)$ does not effect the sign i.e. $\tilde{G}_N(t):=(-1)^NG_N(-\cosh(t))\sinh(t)$ has the same sign as $(-1)^NG_N(-\cosh(t))$. Thus, it's equivalent to prove that $\tilde{G}_N(t)$ is positive for all $t\in(0,+\infty)$. We also note that $\tilde{G}_N(t)$ will be a linear combination of exponential functions $e^{m t}$ with integer $m$'s, because of the nature of hyperbolic functions. The exponential with smallest exponent is $e^{-(N+1)t}$, which is due to the $I_N\cosh(Nt)\sinh(t)$ term that appears in $\tilde{G}_N$, hence we multiply $\tilde{G}_N(t)$ again by $e^{(N+1)t}$ (this operation doesn't affect the sign either because exponential functions are positive) so that we can express $\tilde{G}_N(t)e^{(N+1)t}$ as a polynomial in $e^{t}$. We compute : \begin{align*}
    \tilde{G}_N(t)e^{(N+1)t}&=-I_{N+1}\sinh(Nt)e^{(N+1)t}+I_N\sinh((N-1)t)e^{(N+1)t}\\
    &-I_N(-1)^N\sinh(t)e^{(N+1)t}+I_N\cosh(Nt)\sinh(t)e^{(N+1)t}\\
    &=-I_{N+1}\frac{e^{(2N+1)t}-e^{t}}{2}+I_N\frac{e^{2Nt}-e^{2t}}{2}\\
    &-I_N(-1)^N\frac{e^{(N+2)t}-e^{Nt}}{2}+I_N\frac{e^{(2N+2)t}-e^{2Nt}+e^{2t}-1}{4}
\end{align*}
To finally conclude positivity we need to have control over the relative size of $I_{N+1}$ and $I_{N}$, because the highest order term $e^{(2N+2)t}$ has coefficient $I_N$ while the second highest has coefficient $-I_{N+1}$. Fortunately, we have the estimate $\frac{I_{N+1}}{I_N}\leq \frac{4}{5(N+1)}$ which turns out to be crucial later. Let's transform the equality into the following way : \begin{align*}
    & 4\frac{\tilde{G}_N(t)e^{(N+1)t}}{I_N}\\
    =&e^{(2N+2)t}-2\frac{I_{N+1}}{I_N}e^{(2N+1)t}+e^{2Nt}-2(-1)^Ne^{(N+2)t}+2(-1)^Ne^{Nt}\\
    &-e^{2t}+2\frac{I_{N+1}}{I_N}e^{t}-1\\
    =&(e^{2Nt}-1)\left(e^{2t}-2\frac{I_{N+1}}{I_{N}}e^{t}+1\right)-2(-1)^{N}e^{Nt}(e^{2t}-1)
\end{align*}
Now we divide both sides by $e^t-1$ which is positive for $t>0$ and use the formula $1+r+\dots+r^n=\frac{r^{n+1}-1}{r-1}$, we get 
\begin{subequations}
    \begin{align}
        \frac{4\tilde{G}_N(t)e^{(N+1)t}}{I_N(e^t-1)}&=\frac{e^{2Nt}-1}{e^t-1}\left(e^{2t}-2\frac{I_{N+1}}{I_{N}}e^{t}+1\right)-2(-1)^Ne^{Nt}(e^t+1)\\
        &=\sum_{k=0}^{2N-1}e^{tk}\left(e^{2t}-2\frac{I_{N+1}}{I_{N}}e^{t}+1\right)-2(-1)^Ne^{Nt}(e^t+1)\\
        &=A+B+C+D+E
    \end{align}
\end{subequations}
Where $A,B,C,D,E$ are \begin{align}
    A&=\sum_{k=0,k\notin\{N-2,N-1,N,N+1\}}^{2N-1}e^{tk}\left(e^{2t}-2\frac{I_{N+1}}{I_{N}}e^{t}+1\right)\\
    B&=e^{(N+1)t}\left(e^{2t}-2\frac{I_{N+1}}{I_{N}}e^{t}+(1-a)\right)\\
    C&=e^{Nt}\left(e^{2t}+\left(-2\frac{I_{N+1}}{I_{N}}-b\right)e^{t}+1\right)\\
    D&=e^{(N-1)t}\left(e^{2t}+\left(-2\frac{I_{N+1}}{I_{N}}-c\right)e^{t}+1\right)\\
    E&=e^{(N-2)t}\left((1-d)e^{2t}-2\frac{I_{N+1}}{I_{N}}e^{t}+1\right)
\end{align}
with $a+b=2(-1)^N$,$c+d=2(-1)^N$ which are some parameters that we will choose later. Remark that from $A$, we exclude the summing terms with $k\in\{N-2,N-1,N,N+1\}$ and $B,C,D,E$ correspond to the situation where $k\in\{N+1,N,N-1,N-2\}$ and they absorb into them the term $-2(-1)^Ne^{Nt}(e^t+1)$ that appeared in the original representation of $\frac{4\tilde{G}_N(t)e^{(N+1)t}}{I_N(e^t-1)}$. Now we should tune the parameters $a,b,c,d$ correctly so that each of the following quadratic terms
\begin{subequations}\label{eqn:quad_et}
    \begin{align}
        &e^{2t}-2\frac{I_{N+1}}{I_{N}}e^{t}+1,\\
        &e^{2t}-2\frac{I_{N+1}}{I_{N}}e^{t}+(1-a),\\
        &e^{2t}+\left(-2\frac{I_{N+1}}{I_{N}}-b\right)e^{t}+1,\\
        &e^{2t}+\left(-2\frac{I_{N+1}}{I_{N}}-c\right)e^{t}+1,\\
        &(1-d)e^{2t}-2\frac{I_{N+1}}{I_{N}}e^{t}+1,
    \end{align}    
\end{subequations}
is positive for $t>0$, then $A,B,C,D,E$ will all be positive and thus $\frac{4\tilde{G}_N(t)e^{(N+1)t}}{I_N(e^t-1)}$ will be positive and this will imply the positivity of $(-1)^NG_N(x)$. 

Consider $a=d=2\frac{I_{N+1}}{I_N}$ and $b=c=2(-1)^N-2\frac{I_{N+1}}{I_{N}}$, let's analyze the quadratic functions \eqref{eqn:quad_et} in $e^t$:\begin{enumerate}
    \item For $e^{2t}-2\frac{I_{N+1}}{I_{N}}e^{t}+1$, the discriminant is $\Delta=4\left(\frac{I_{N+1}}{I_{N}}\right)^2-4$ since $\frac{I_{N+1}}{I_{N}}\leq \frac{4}{5(N+1)}\rightarrow0,N\rightarrow\infty$, $\Delta<0$ eventually when $N$ is large enough therefore $e^{2t}-2\frac{I_{N+1}}{I_{N}}e^{t}+1>0,\forall t>0$ for large enough $N$;
    \item For $e^{2t}-2\frac{I_{N+1}}{I_{N}}e^{t}+(1-a)$ and $(1-d)e^{2t}-2\frac{I_{N+1}}{I_{N}}e^{t}+1$, the discriminants are the same and are equal to $\Delta=4\left(\frac{I_{N+1}}{I_{N}}\right)^2-4\left(1-2\frac{I_{N+1}}{I_N}\right)<0$, eventually for large enough $N$. Therefore \eqref{eqn:quad_et} (b) and (e) are positive for all $t>0$ when eventually $N$ is large enough;
    \item With the choice $b=c=2(-1)^N-2\frac{I_{N+1}}{I_{N}}$, $e^{2t}+\left(-2\frac{I_{N+1}}{I_{N}}-b\right)e^{t}+1=e^{2t}-2(-1)^Ne^t+1=(e^t-(-1)^N)^2\geq0,\forall t>0$. Therefore \eqref{eqn:quad_et} (c) and (d) are positive for all $t>0$.
\end{enumerate}
As a result $\frac{4\tilde{G}_N(t)e^{(N+1)t}}{I_N(e^t-1)}=A+B+C+D+E>0,\forall t>0$ when eventually $N$ is large enough.

Therefore, we can conclude the desired positivity which holds for large enough $N$, an estimate of the smallest such $N$ can be found by keeping track of the locations where we deduce negativity of discriminants by using the estimate $\frac{I_{N+1}}{I_N}\leq \frac{4}{5(N+1)}$. Let's derive the least $N$, there are two discriminants that involve $\frac{I_{N+1}}{I_N}$ and let's keep in mind that $\frac{I_{N+1}}{I_N}\in (0,\frac{4}{5(N+1)}]$:\begin{enumerate}
    \item $\Delta=4\left(\frac{I_{N+1}}{I_N}\right)^2-4$ : \begin{align*}
        & 4\left(\frac{I_{N+1}}{I_N}\right)^2-4 < 0\\
        \impliedby&\frac{I_{N+1}}{I_N} < 1 \impliedby \frac{4}{5(N+1)}<1\\
        \impliedby&N\geq 1
    \end{align*}
    \item $\Delta=4\left(\frac{I_{N+1}}{I_{N}}\right)^2-4\left(1-2\frac{I_{N+1}}{I_N}\right)$ : \begin{align*}
        & 4\left(\frac{I_{N+1}}{I_{N}}\right)^2-4\left(1-2\frac{I_{N+1}}{I_N}\right) < 0\\
        \impliedby&\frac{I_{N+1}}{I_N} \in (-1-\sqrt{2},-1+\sqrt{2})\\
        \impliedby& (0,\frac{4}{5(N+1)}]\subset (-1-\sqrt{2},-1+\sqrt{2})\\
        \impliedby& \frac{4}{5(N+1)} < -1+\sqrt{2} \impliedby N > \frac{4}{5(-1+\sqrt{2})}-1\approx 0.93
    \end{align*}
\end{enumerate}
Since the estimates above hold for $N\geq 1$ the result follows.
\end{proof}

\bibliographystyle{amsplain}



\providecommand{\bysame}{\leavevmode\hbox to3em{\hrulefill}\thinspace}
\providecommand{\MR}{\relax\ifhmode\unskip\space\fi MR }
\providecommand{\MRhref}[2]{%
  \href{http://www.ams.org/mathscinet-getitem?mr=#1}{#2}
}
\providecommand{\href}[2]{#2}


\end{document}